\documentclass[10pt]{article}

\usepackage{amsmath, amsthm, amssymb, enumerate}
\usepackage{color}
\usepackage{datetime}
\usepackage{fancyhdr}
\usepackage{bbm}
\chardef\coloryes=0 
\chardef\isitdraft=0 

\ifnum\isitdraft=1 
   \def\version{2} 
   \def\eqref#1{({\ref{#1}})}                

\usepackage[color,notcite]{showkeys}

\usepackage[notcite]{showkeys}

\definecolor{labelkey}{gray}{.3}

\definecolor{refkey}{rgb}{.3,0.3,0.3}

\else

  \def\startnewsection#1#2{\section{#1}\label{#2}\setcounter{equation}{0}}   
  \def\nnewpage{} 
\fi

\begin{document}
\def\ques{{\colr \underline{??????}\colb}}
\def\nto#1{{\colC \footnote{\em \colC #1}}}
\def\fractext#1#2{{#1}/{#2}}
\def\fracsm#1#2{{\textstyle{\frac{#1}{#2}}}}   
\def\nnonumber{}


\def\colr{{}}
\def\colg{{}}
\def\colb{{}}
\def\cole{{}}
\def\colA{{}}
\def\colB{{}}
\def\colC{{}}
\def\colD{{}}
\def\colE{{}}
\def\colF{{}}

\ifnum\coloryes=1

  \definecolor{coloraaaa}{rgb}{0.1,0.2,0.8}
  \definecolor{colorbbbb}{rgb}{0.1,0.7,0.1}
  \definecolor{colorcccc}{rgb}{0.8,0.3,0.9}
  \definecolor{colordddd}{rgb}{0.0,.5,0.0}
  \definecolor{coloreeee}{rgb}{0.8,0.3,0.9}
  \definecolor{colorffff}{rgb}{0.8,0.3,0.9}
  \definecolor{colorgggg}{rgb}{0.5,0.0,0.4}

 \def\colg{\color{colordddd}}
 \def\colb{\color{black}}
 \def\colr{\color{red}}
 \def\cole{\color{colorgggg}}

 \def\colA{\color{coloraaaa}}
 \def\colB{\color{colorbbbb}}
 \def\colC{\color{colorcccc}}
 \def\colD{\color{colordddd}}
 \def\colE{\color{coloreeee}}
 \def\colF{\collor{colorffff}}
 \def\colG{\color{colorgggg}}

\fi
\ifnum\isitdraft=1
   \chardef\coloryes=1 
   \baselineskip=17pt
\pagestyle{myheadings}
\reversemarginpar

\def\const{\mathop{\rm const}\nolimits}  
\def\diam{\mathop{\rm diam}\nolimits}    

 \def\llabel#1{\label{#1}{\ \mbox{\rm\color{red} {#1}\color{black}}}}

\def\rref#1{{\ref{#1}{\rm \tiny \fbox{\tiny #1}}}}
\def\theequation{\fbox{\bf \thesection.\arabic{equation}}}
\def\ccite#1{{\cite{#1}{\rm \tiny ({#1})}}}

\def\startnewsection#1#2{\newpage\cog \section{#1}\cob\label{#2}

\setcounter{equation}{0}
\pagestyle{fancy}

\lhead{\cob Section~\ref{#2}, #1 }
\cfoot{}
\rfoot{\thepage}
\lfoot{\cob{\today,~\currenttime}~(c75-iklt2, Version~\fbox{\version})}}
\chead{}
\rhead{\thepage}
\def\nnewpage{\newpage}

\newcounter{startcurrpage}
\newcounter{currpage}

\def\llll#1{{\rm\tiny\fbox{#1}}}
   \def\blackdot{{\color{red}{\hskip-.0truecm\rule[-1mm]{4mm}{4mm}\hskip.2truecm}}\hskip-.3truecm}
   \def\bdot{{\colC {\hskip-.0truecm\rule[-1mm]{4mm}{4mm}\hskip.2truecm}}\hskip-.3truecm}
   \def\purpledot{{\colA{\rule[0mm]{4mm}{4mm}}\colb}}
   \def\pdot{\purpledot}
\else
   \baselineskip=15pt
   \def\blackdot{{\rule[-3mm]{8mm}{8mm}}}
   \def\purpledot{{\rule[-3mm]{8mm}{8mm}}}
   \def\pdot{}
\fi

\def\tdot{\fbox{\fbox{\bf\tiny I'm here; \today \ \currenttime}}}
\def\nts#1{{\hbox{\bf ~#1~}}} 
\def\nts#1{{\colr\hbox{\bf ~#1~}}} 
\def\ntsf#1{\footnote{\hbox{\bf ~#1~}}} 
\def\ntsf#1{\footnote{\colr\hbox{\bf ~#1~}}} 
\def\bigline#1{~\\\hskip2truecm~~~~{#1}{#1}{#1}{#1}{#1}{#1}{#1}{#1}{#1}{#1}{#1}{#1}{#1}{#1}{#1}{#1}{#1}{#1}{#1}{#1}{#1}\\}
\def\biglineb{\bigline{$\downarrow\,$ $\downarrow\,$}}
\def\biglinem{\bigline{---}}
\def\biglinee{\bigline{$\uparrow\,$ $\uparrow\,$}}

\newtheorem{Theorem}{Theorem}[section]
\newtheorem{Corollary}[Theorem]{Corollary}
\newtheorem{Proposition}[Theorem]{Proposition}
\newtheorem{Lemma}[Theorem]{Lemma}
\newtheorem{Remark}[Theorem]{Remark}
\newtheorem{definition}{Definition}[section]
\def\theequation{\thesection.\arabic{equation}}
\def\endproof{\hfill$\Box$\\}
\def\square{\hfill$\Box$\\}
\def\comma{ {\rm ,\qquad{}} }            
\def\commaone{ {\rm ,\qquad{}} }         
\def\dist{\mathop{\rm dist}\nolimits}    
\def\Lip{\mathop{\rm Lip}\nolimits}    
\def\sgn{\mathop{\rm sgn\,}\nolimits}    
\def\Tr{\mathop{\rm Tr}\nolimits}    
\def\div{\mathop{\rm div}\nolimits}    
\def\supp{\mathop{\rm supp}\nolimits}    
\def\divtwo{\mathop{{\rm div}_2\,}\nolimits}    
\def\re{\mathop{\rm {\mathbb R}e}\nolimits}    
\def\indeq{\qquad{}}                     
\def\period{.}                           
\def\semicolon{\,;}                      

\def\HH{{\mathbb H}}
\def\VV{{\mathbb V}}


\title{On the Galerkin approximation and strong norm bounds for the stochastic Navier-Stokes equations with multiplicative noise}
\author{Igor Kukavica, Kerem U\u{g}urlu, and Mohammed Ziane}
\maketitle

\date{}

\begin{center}
\end{center}

\medskip

\indent Department of Mathematics, University of Southern California, Los Angeles, CA 90089\\
\indent e-mails: kukavica@usc.edu, kugurlu@usc.edu, ziane@usc.edu

\begin{abstract}
We investigate the convergence of the Galerkin approximation for the
stochastic Navier-Stokes equations in an
open  bounded domain $\mathcal{O}$ with the non-slip boundary
condition. 
We prove that
  \begin{equation*}
   \mathbb{E} \left[ \sup_{t \in [0,T]} \phi_1(\lVert  (u(t)-u^n(t))
      \rVert^2_V) \right] \rightarrow 0
  \end{equation*}
as $n \rightarrow \infty$
for any deterministic time $T > 0$ and for a
specified moment function $\phi_1(x)$ 
where
$u^n(t,x)$ denotes  the Galerkin approximation of the solution
$u(t,x)$. Also, we 
provide a result on uniform boundedness of the moment
$\mathbb{E} [ \sup_{t \in [0,T]} \phi(\lVert  u(t)
\rVert^2_V) ] $
where $\phi$  grows as a single logarithm at infinity.
Finally, we summarize results on convergence of the Galerkin
approximation 
up to a deterministic time $T$
when the $V$-norm is replaced by
the $H$-norm.

%
%
\end{abstract}

\noindent\thanks{\em Mathematics Subject Classification\/}:
35Q30, 60H15,
60H30\\
\noindent\thanks{\em Keywords:\/}
stochastic Navier-Stokes equations,
moment estimates,
Galerkin approximation

\startnewsection{Introduction}{sec1}
In this paper, we address the convergence properties of the Galerkin
approximation to the stochastic Navier-Stokes equations and obtain new
estimates on the convergence in the strong norm.

The stochastic Navier-Stokes equations (SNSE) in a smooth bounded
domain $\mathcal{O}\subseteq \mathbb{R}^2$ with a multiplicative white
noise read
  \begin{align}
   \label{main_eqn}
   &d u +
    \bigl(
     (u \cdot \nabla)u - \nu \Delta u + \nabla p 
    \bigr)
    dt
    = f dt + g(u)d\mathcal{W}
   \nonumber\\&
   \nabla \cdot u = 0
   \nonumber\\&
   u(0) = u_0
  \end{align}
\cite{BKL00,CG94,C89,CP97,DD03,FG95,FR02,GV14,M02,MR2,MS02,O06,S03}.
We consider the
Dirichlet boundary condition $u = 0$ on $\partial \mathcal{O}$.  
Here $u = (u_1, u_2)$ represents the velocity field, $p$ represents
the pressure, and $\nu$ is the viscosity, whereas $f$ stands
for the deterministic force. Also, 
$g(u)\mathcal{W} = \sum_k g_k(u)e_kW_k$ 
stands for an
infinite dimensional Brownian motion, where
each $W_k$ is the standard one dimensional Brownian motion and
$g_k(u)$ are the corresponding Lipschitz coefficients.

The study of the SNSE
 was initiated by Bensoussan and Temam in 1973 \cite{BT73}, and 
the equations have been extensively studied since then
(\cite{BF00,F08,GTW,K06,KS12,PR07}). 
The well-posedness in $L^2$ was considered by Breckner
\cite{B00}, while the existence
in Sobolev spaces $W^{1,p}$, where 
$p>2$, was obtained by Brzezniak and Peszat \cite{BP00} as well as by
Mikulevicius and Rozovsky \cite{MR1}.
Finally, the local existence in $H^{1}$ 
was proven 
in \cite{GZ09},
where a method was introduced which extends
also to less regular Sobolev spaces.
For a comprehensive treatment of the
subject of SNSE,
we refer the
reader to the books by Vishik and Fursikov \cite{FV88}, Capinski and
Cutland \cite{CC95}, and Flandoli \cite{F08}.

As in the case of the deterministic NSE, the solutions are commonly
constructed as a limit of solutions of the Galerkin system
\cite{BS07,CF85,T01}.
In \cite{B00}, Breckner proved that the solutions $u$ of the 
SNSE
can be approximated by solutions $u^n$ of the corresponding
Galerkin systems.
Namely, she proved that for all $t>0$, we have
  \begin{equation}
   {\mathbb E}
    \left[
      \Vert u(t)-u^n(t)\Vert_{H}^2
       +
      \int_{0}^{t} \Vert u(s)-u^n(s)\Vert_{V}^2 \,ds
    \right]
   \to 0
   \label{EQ13}
  \end{equation}
as $n\to \infty$
(cf.~\eqref{EQ19} and \eqref{EQ20} for the definitions of the spaces
$H$ and $V$).
In the absence of boundaries,
her results extend easily to the case of stronger norms.
More specifically, using the cancellation property
  \begin{equation}
    (B(u,u),A u)=0   
   \label{EQ15}
  \end{equation}
where $B$ is the bilinear form and $A$ the Stokes
operator, which is valid 
in the case of \emph{periodic} boundary conditions, one can easily obtain a stronger
convergence result
  \begin{equation}
   {\mathbb E}
    \left[
      \Vert u(t)-u^n(t)\Vert_{V}^2
       +
      \int_{0}^{t} \Vert u(s)-u^n(s)\Vert_{H^2}^2 \,ds
    \right]
   \to 0
   \label{EQ14}
  \end{equation}
as $n\to\infty$, under suitable assumptions on the noise.

The goal of this paper is to address
the convergence of the Galerkin approximation pointwise in time for the $V$ norm
in the case of the Dirichlet boundary conditions
when the cancellation property 
\eqref{EQ15} does not hold.
In this case, it is easy to obtain results in this
direction
up to a suitable stopping time. However, the
finiteness of the expected value of the second moment of the norm
$\Vert u(t)\Vert_{V}^2$ for any fixed {\em non-random} time $t$ is an
open problem. 
By the same token, it is not known whether
the expected value of the supremum of
$\Vert u(t)-u^n(t)\Vert_{V}^2$ 
up to a deterministic time
converges to 0 as $n\to \infty$.
A positive result in this direction was
obtained in \cite{KV}, where it
was proven that
  \begin{equation}
   {\mathbb E}
   \left[
    \sup_{0\le t\le T}
     \tilde    \phi(\Vert u\Vert_{V}^2)
   \right]
   <\infty
   \label{EQ16}
  \end{equation}
where
  \begin{equation}
   \tilde \phi(\tau)
   = \log (1 + \log (1 +\tau))
   \comma \tau\in(0,\infty)
   \period
   \label{EQ17}
  \end{equation}

The aim of this paper is twofold.
First, we strengthen the main result
in \cite{KV} by showing that \eqref{EQ16}
holds with
  \begin{equation}
   \phi(\tau)
   = \log (1+\tau)
   \label{EQ18}
  \end{equation}
instead of $\tilde\phi$ (cf.~Theorem~\ref{T02} below). 
The second goal is
to obtain the convergence of the Galerkin approximation in the $V$ norm.
Namely, we prove that
  \begin{equation}
   \mathbb{E} 
    \left [ \sup_{[0,T]} \phi({\lVert u - u^n \rVert_{V}^2})^{1-\epsilon}
    \right ]\rightarrow 0
  \end{equation}
as $n\to\infty$ for all $\epsilon>0$.


The paper is organized as follows. In Section~\ref{sec2}, we give the
theoretical background along with the deterministic and the stochastic
settings.
In Section~\ref{sec3} we state the main results on the convergence of
the Galerkin approximations 
in the $V$-norm  and on the finiteness of the logarithmic moment functions.
In Remark~\ref{R01}, we summarize results on convergence when the
$V$-norm is replaced by the $H$-norm.
Finally, Section~\ref{S4} contains the proof of the convergence of the Galerkin
approximation to the original solutions. 
The proof uses the new moment
estimate provided in Theorem~\ref{T02}.

\startnewsection{Functional Setting}{sec2}
First, we recall the deterministic and the probabilistic frameworks 
used throughout the paper. 
\subsection{Deterministic Framework}
Let $\mathcal{O}$ be a smooth bounded open connected subset of $\mathbb{R}^2$,
and let $\mathcal{V} = \{u \in C_0^{\infty}(\mathcal{O}): \nabla \cdot u = 0 \}$. Denote by $H$ and $V$ the closures of $\mathcal{V}$ in $L^2(\mathcal{O})$ and $H^{1}({\mathcal O})$ respectively. The spaces $H$ and $V$ are identified by
  \begin{align}
   \label{EQ19}
   H &= \{ u \in L^2(\mathcal{O}): \nabla \cdot u = 0, u \cdot N |_{\partial \mathcal{O}} = 0 \}, \\
   V &= \{ u \in H_0^1(\mathcal{O}): \nabla \cdot u = 0 \}
   \label{EQ20}
  \end{align}
(cf.~\cite{CF88,T01}).
Here $N$ is the outer pointing normal to $\partial \mathcal{O}$. On $H$ we denote the $L^2(\mathcal{O})$ inner product and the norm  as 
  \begin{align}
   &\langle u,v \rangle = \int_{\mathcal{O}} u \cdot v dx
   \nonumber\\&
   \lVert u \rVert_H = \sqrt{\langle u,u \rangle}
   \period
  \end{align}
Let $\mathcal{P}_H$ be the Leray-Hopf projector of $L^2(\mathcal{O})$
onto $H$. Recall that for $u \in L^2(\mathcal{O})$ we have
$\mathcal{P}_H u = (1 - \mathcal{Q}_H)u$ where 
$\mathcal{Q}_H u = \nabla \pi_1+\nabla\pi_2 $ 
and $\pi_1,\pi_2 \in H^1(\mathcal{O})$ are solutions of the problems
  \begin{align}
   &\Delta\pi_1 = \nabla \cdot u  \mbox{ in }  \mathcal{O}
   \nonumber\\&
   \pi_1 = 0 \mbox{ on }  \partial\mathcal{O}
  \end{align}
and
  \begin{align}
   &\Delta\pi_2 = 0  \mbox{ in }  \mathcal{O}
   \nonumber\\&
   \nabla \pi_2 \cdot N = u-\nabla \pi_1 \mbox{ on }  \partial\mathcal{O}
  \end{align}
(\cite{CF88,T01}).
Let
  \begin{equation}
   A = -\mathcal{P}_H \Delta
  \end{equation}
be the Stokes operator with the domain $\mathcal{D}(A)= V \cap H^2(\mathcal{O})$. The dual of $V = \mathcal{D}(A^{1/2})$ with respect to $H$ is denoted by $V'= \mathcal{D}(A^{-1/2})$. Here $A$ is defined as a bounded, linear map from $V$ to $V'$ via 
  \begin{equation*}
   \langle Au, v \rangle = \int_{\mathcal{O}} \nabla u \cdot \nabla v dx
   \comma u,v \in V,
  \end{equation*} 
with the corresponding norm defined as 
  \begin{equation*}
   \lVert u \rVert_V^2 = \langle Au, u \rangle = \langle A^{1/2}u,A^{1/2}u \rangle
   \comma u \in V
   \period
  \end{equation*}
By the theory of symmetric, compact operators applied to $A^{-1}$, there exists an orthonormal basis $\{e_k\}$ for $H$ consisting of eigenfunctions of $A$. 
The corresponding eigenvalues $\{ \lambda_k \}$ form an increasing, unbounded sequence 
  \begin{equation*}
   0 < \lambda_1 \leq \lambda_2 \leq \cdots \leq \lambda_n \leq \cdots
  \period
\end{equation*}
We also define the nonlinear term as a bilinear mapping $ V \times V$ to $V'$ via 
  \begin{equation*}
   B(u,v)= \mathcal{P}_H(u \cdot \nabla v )
   \period
  \end{equation*} 
The deterministic force $f$ is assumed to be bounded with values in $H$.
Note that the cancellation property $\langle B(u,v),v \rangle = 0$ holds for $u,v \in V$. 

\subsection{Stochastic Framework}
In this section, we recall the necessary background material for stochastic analysis in infinite dimensions needed in this paper (cf.~\cite{DZ92, DGT11, F08, PR07}). Fix a stochastic basis $\mathcal{S}= ( \Omega, \mathcal{F}, \mathbb{P}, \{ \mathcal{F}_t\}, \mathcal{W} ) $, which consists of a complete probability space ($\Omega,\mathbb{P}$), equipped with a complete right-continuous filtration $\mathcal{F}_t$, and a cylindrical Brownian motion $\mathcal{W}$, defined on a separable Hilbert space $U$ adapted to this filtration. 

Given a separable Hilbert space $X$, 
we denote by $L_2(U,X)$ the space of 
Hilbert-Schmidt operators from $U$ to $X$, 
equipped with the norm 
$\lVert G \rVert_{L_2(U,X)}= (\sum_k \lVert G \rVert^2_X)^{1/2}$ (cf.~\cite{DZ92}). 
For an $X$-valued predictable process 
$G \in L^2( \Omega;L^2_{\text{loc}}([0,\infty]);L_2(U,X) )$, we define the It\^o stochastic integral
  \begin{equation}
   \label{def1}
   \int_0^t G d\mathcal{W} = \sum_k \int_0^t G_k dW_k
  \end{equation}
which lies in the space $\mathcal{O}_X$ of $X$-valued square integrable martingales. We also recall the Burkholder-Davis-Gundy inequality: For all $p \geq 1$ we have 
  \begin{equation}
   \label{def2}
   \mathbb{E} \left [ \sup_{t \in [0,T]} \left\Vert \int_0^t G d\mathcal{W} \right\Vert^p_X \right ] 
   \leq 
   C\mathbb{E} \left [ \int_0^T  \rVert  G \lVert^2_{L_2(U,X)} \right]^{p/2} 
  \end{equation}
for some $C=C(p) > 0$.
\subsection{Conditions on the noise}
Given a pair of Banach spaces $X$ and $Y$, we denote by $\Lip_u(X,Y)$ the collection of continuous functions $h \colon [0,\infty) \times X \rightarrow Y$ which are sublinear 
  \begin{equation}
   \label{def3}
   \lVert h(t,x) \rVert_Y \leq K_Y(1 + \rVert x \lVert_X)\comma t \geq 0, x \in X  
  \end{equation}
and Lipschitz 
  \begin{equation}
   \label{def4}
   \lVert h(t,x) - h(t,y) \rVert_Y \leq K_Y \lVert x - y \rVert_X \comma t \geq 0\commaone x,y \in X
  \end{equation}
for some constant $K_Y > 0$ independent of $t$.
The noise term $g(u)d\mathcal{W}$, which is defined by
  \begin{equation}
   \label{def5}
   g = \{g_k\}_{k \geq 1}\colon [0,\infty) \times H \rightarrow L_2(U,H)
  \end{equation}
satisfies
  \begin{equation}
   \lVert g(t,x) \rVert_{L_2(U,\mathcal{D}(A^{j/2}))} \leq K_j(1 + \lVert x \rVert_{\mathcal{D}(A^{j/2})})\textrm{ for } j \in \{0,1,2\} 
  \end{equation}
and
  \begin{equation}
   \lVert g(t,x) - g(t,y) \rVert_{L_2(U,\mathcal{D}(A^{j/2}))} \leq K_j \lVert x  - y \rVert_{\mathcal{D}(A^{j/2})}\textrm{ for } j \in \{0,1,2\} 
   \period
  \end{equation}
In particular, we have
  \begin{equation}
   \label{g_con}
   g \in \Lip_u(H, L^2(U,H)) \cap \Lip_u(V,L_2(U,V)) \cap \Lip_u(\mathcal{D}(A),L_2(U,\mathcal{D}(A)))
   \period
  \end{equation}
As in \cite{KV},
denote
${\HH}=L^2([0,T];H)$
and
${\VV}=L^2([0,T];V)$.
Given $u \in L^2(\Omega; L^2([0,T];H))$ and with $g$ as above, the stochastic integral $\int_0^t g(u)d\mathcal{W}$ is a well-defined $H$-valued It\^o stochastic integral that is predictable and is such that 
  \[
   \left \langle \int_0^t g(u)d\mathcal{W}, v \right \rangle = \sum_k \int_0^t \langle g_k(u),v \rangle d\mathcal{W}_k
  \]
holds for all $v \in H$.

\subsection{Notion of a Solution}
We consider \emph{strong pathwise solutions} in the PDE sense, i.e., solutions bounded in time with values in $V$, square integrable in time with values in $\mathcal{D}(A)$, and \emph{strong} in the probabilistic sense, i.e., the driving noise and the filtration are given in advance.

\begin{definition} \label{D01}
{\rm 
Let $g$ be as in \eqref{g_con} predictable, and let $f \in L^1(\Omega;L^4([0,T);V'))$ be predictable. Assume that the initial data $u_0 \in L^4(\Omega;H) \cap L^2(\Omega;V)$ is $\mathcal{F}_0$ measurable. The pair $(u,\tau)$ is called a pathwise strong solution of the system if $\tau$ is a strictly positive stopping time, $u(\cdot \wedge \tau )$ is a predictable process in $H$ such that 
  \begin{equation}
   \label{init_ass}
   u(\cdot \wedge \tau ) \in L^2(\Omega;C([0,T];V))
  \end{equation}
with
  \begin{equation}
   u\mathbf{1}_{t\leq \tau} \in L^2(\Omega;L^2([0,T];\mathcal{D}(A)))
  \end{equation}
and if
  \begin{equation}
   \langle u(t \wedge \tau), v \rangle + \int_0^{t \wedge \tau} \bigl\langle \nu Au + B(u,u) - f, v \bigr\rangle dt = \bigl\langle u_0,v\bigr\rangle + \sum_k\int_0^{t \wedge \tau} \langle g_k(u),v \rangle dW_k
  \end{equation}
holds for every $v \in H$. Moreover, $(u,\xi)$ is called a maximal pathwise strong solution if $\xi$ is a strictly positive stopping time and there exists a non-decreasing sequence
of  stopping times $\tau_n$ such that 
$\tau_n \rightarrow \xi$ and $(u, \tau_n)$ is a local strong solution and 
  \begin{equation}
   \label{EQ11}
   \sup_{t \in [0,\tau_n]} \lVert u \rVert^2_V + \nu\int_0^{\tau_n} \lVert Au \rVert^2_H dt \geq n
  \end{equation}
on the set $\{ \xi < \infty \}$. Such a solution is called global if $\mathbb{P}(\xi < \infty ) = 0$.
}
\end{definition}

We proceed with the definition of the Galerkin system.

\begin{definition}
\label{D02}
{\rm
An adapted process $u^n$ in $C([0,T];H_n)$,
where  $H_n={\mathcal L}\{e_1,\ldots,e_n\}$,
is a solution to the Galerkin system of order $n$ if for any $v$ in $H_n$ 
   \begin{align}
   &d\langle u^n,v \rangle + \langle \nu Au^n + B(u^n),v \rangle dt = \langle f,v \rangle dt + \sum_{k=1}^\infty \langle g_k(u^n), v \rangle dW_k
   \nonumber\\&
   \langle u^n(0),v \rangle = \langle u_0,v \rangle
   \period
   \label{galerkin_long}
   \end{align}
We may also rewrite  \eqref{galerkin_long} as equations in $H_n$, i.e.,
  \begin{align}
   &du^n + (\nu A u^n + P_n B (u^n))dt = P_n f dt + \sum_{k=1}^\infty P_n g_k (u^n)dW_k
   \nonumber\\&
   u^n(0) = P_n u_0 = u_0^n
   \period
   \label{galerkin_short}
  \end{align}
}
\end{definition}

\startnewsection{The Main Results}{sec3}
Our main result establishes the convergence of Galerkin approximations
in the $V$ norm up to any deterministic time $T$.

\cole
\begin{Theorem}
\label{T01}
Let $\epsilon \in (0,1)$ and let $T>0$ be arbitrary. Suppose that $u$ is a solution to the equation~\eqref{main_eqn}, and let $u^n$ be the corresponding Galerkin approximation.
Then we have 
  \begin{equation}
   \mathbb{E} \left [ \sup_{[0,T]} \phi_1({\lVert u - u^n \rVert_{V}^2}) \right ]\rightarrow 0
  \end{equation}
as $n \rightarrow \infty$,
where $\phi_1(x) = (\log ( 1 + x))^{1-\epsilon}$,

\end{Theorem}
\colb


The main tool used in the proof is the following improvement of the main
result in \cite{KV} of independent interest.

\cole
\begin{Theorem} 
\label{T02}
Let $u_0$, $f$, and $g$ be as in Definition~\ref{D01} and 
suppose that $u$ is the solution to the equation~\eqref{main_eqn}. Then we have 
  \begin{equation}
   \mathbb{E}\left[\sup_{[0,T]} \phi(\lVert u \rVert_V^2)\right] \leq C(f,g,u_0,T),
   \label{EQ10}
  \end{equation}
where $\phi(x) = \log(1 + x)$.
\end{Theorem}
\colb




\begin{Remark}
\label{R01}
\rm
When considering the convergence of the Galerkin 
approximations $u^n$ in $H$, a
stronger results may be obtained.
Namely, let $u$ be the solution to the equation~\eqref{main_eqn}, 
and let $u^n$ be the corresponding Galerkin approximation. Assume that
$f \in L^{2k}(\Omega;L^{2k}([0,\infty);V'))$ 
and 
$u_0 \in L^{2k+2}(\Omega;H) \cap L^2(\Omega;V)$
for all $k\in{\mathbb N}$.
Then we have
  \begin{equation}
   \mathbb{E} \left [ \sup_{[0,T]} \lVert u - u^n \rVert_{H}^{m} \right ] \rightarrow 0
    \hbox{~~as $n\to\infty$}
   \comma m\in{\mathbb N}
   \label{EQ22}
  \end{equation}
for any deterministic time $T>0$.
Indeed, let $k\in{\mathbb N}$.
By \cite{FG95}, we have
  \begin{align} 
   \mathbb{E}\left[\sup_{[0,T]} \lVert u \rVert_H^{2k} \right] 
   +
   \mathbb{E}\left[ \int_0^T \lVert u \rVert_V^2\lVert u \rVert_H^{2k-2} ds\right] &\leq C(k,\lVert u_0 \rVert_H^{2k}, \lVert f \rVert_{V'}^{2k}, T)
   \label{EQ24}
   \period
  \end{align}
Also, by the same argument applied to the Galerkin system, we get
  \begin{align} 
   \mathbb{E}\left[\sup_{[0,T]} \lVert u^n \rVert_H^{2k} \right] 
   +
   \mathbb{E}\left[ \int_0^T \lVert u^n \rVert_V^2\lVert u^n \rVert_H^{2k-2} ds\right] &\leq C(k,\lVert u_0 \rVert_H^{2k}, \lVert f \rVert_{V'}^{2k}, T)
   \period
   \label{EQ25}
  \end{align}
Then, we have using 
  \begin{equation}
    \log ( 1 + x ) \leq x\comma x \geq 0   
   \label{EQ21}
  \end{equation}
Recall that, by \cite{B00}, we have
  \begin{equation}
   \mathbb{P}
         \left(\sup_{t \in [0,T]} \lVert u -u^n \rVert_H \ge \delta\right) 
   \to0
   \label{EQ26}
  \end{equation}
while, by \eqref{EQ24} and \eqref{EQ25},
  \begin{equation}
   \mathbb{E}\left[\sup_{[0,T]} \lVert u - u^n \rVert_H^{2k}\right] \leq 2^{2k}\left(\mathbb{E} \left[\sup_{[0,T]} \lVert u \rVert_H^{2k}\right] + \mathbb{E}\left[\sup_{[0,T]} \lVert u^n \rVert_H^{2k}\right]\right)
   \label{EQ23}
   \period
  \end{equation}
Using the uniform integrability principle
with \eqref{EQ26} and
\eqref{EQ23}, we get
  \begin{equation}
   \mathbb{E}\left[
                   \sup_{[0,T]} \lVert u - u^n \rVert_H^{2k(1-\epsilon)}
             \right] \rightarrow 0
  \end{equation}
as $n \rightarrow \infty$, for every $k\in{\mathbb N}$
and $\epsilon\in(0,1]$, and \eqref{EQ22}
is proven.

It is possible to obtain more precise information 
regarding the convergence in $H$.
Assume first that
  \begin{equation}
   \Vert
    g(t,x)
   \Vert_{{\mathbb H}}
    \le
    C 
   \period
   \label{EQ32}
  \end{equation}
Then estimating
${\mathbb E}[\sup_{[0,T]}\Vert u\Vert_{H}^{2k}]$
for $k=1,2,3,\ldots$ and keeping the dependence on $k$,
we get
  \begin{equation}
   {\mathbb E}\left[\sup_{[0,T]} \exp(\Vert u\Vert_{H}/K)\right]
   \le C
   \label{EQ33}
  \end{equation}
for a sufficiently large constant $K$
(cf.~also \cite[Lemma~3.1]{G} and \cite{KS12} for a different approach).
As in \cite{B00},
we get
  \begin{equation}
   {\mathbb E}\left[\sup_{[0,T]} \exp(\Vert u-u_n\Vert_{H}/K')\right]
   \to 0
   \label{EQ34}
  \end{equation}
as $n\to\infty$,
where $K'$ is any constant larger than $K$.
More generally, if
  \begin{equation}
   \Vert
    g(t,x)
   \Vert_{{\mathbb H}}
    \le
    C 
   (1+ \Vert x\Vert_{{\mathbb H}}^{\alpha})
   \label{EQ35}
  \end{equation}
where $\alpha\in[0,1)$, then instead
  \begin{equation}
   {\mathbb E}\left[\sup_{[0,T]} \exp(\Vert u-u_n\Vert_{H}^{2(1-\alpha)}/K')\right]
   \to 0
  \end{equation}
as $n\to\infty$.
\end{Remark}

\startnewsection{Galerkin Convergence in $V$}{S4}
In this section, we prove
the first main result, Theorem~\ref{T01}.
We first recall the existence result from \cite{GZ09}.

\cole
\begin{Theorem}\cite{GZ09}
\label{stop_time}
Let $\{ u^n \}$ be the sequence of solutions of \eqref{galerkin_long}, 
and let $u$ be the solution to the equation~\eqref{main_eqn}  with $g$, $f$, and $u_0$ as in Definition~\ref{D01}. Then there exists a global, maximal pathwise strong solution $(u,\xi)$. 
Namely, there exists an increasing sequence of strictly positive stopping times $\{\tau_m\}_{m\geq 0}$ 
converging to $\xi$, for which $\mathbb{P}(\xi < \infty) = 0$. Moreover, there exists an increasing sequence of measurable subsets $\{\Omega_s\}_{s \ge 1}$ with  $\Omega_s \uparrow \Omega$ as $s \rightarrow \infty$ 
so that on any $\Omega_s$ we have 
  \begin{equation}
   \lim_{n \rightarrow \infty} \mathbb{E}\left[ \mathbbm{1}_{\Omega_s} \left( \sup_{t \in [0,\tau_m]} \lVert u - u^n \rVert^2_V  + \nu \int_0^{\tau_m} \lVert A(u -u^n) \rVert_H^2 dt \right) \right] = 0
  \end{equation}
as $n\rightarrow \infty$ for any $\tau_m$.
\end{Theorem}
\colb

First, we establish the convergence of the Galerkin approximations in probability.

\cole
\begin{Lemma}\label{V_prob} Let $u$ and $u^n$ be defined as in Definitions~\ref{D01} and~\ref{D02}. 
Then for any deterministic time~$T>0$,
the Galerkin approximations $u^n$ converge in probability 
with respect to the $V$ norm 
to the solution of the equation~\eqref{main_eqn},
i.e.,
for any $\delta>0$ we have
  \begin{equation}
   \mathbb{P} \left( \sup_{t \in [0,T]} \lVert u - u^n \rVert_V^2 \ge \delta \right) \to 0
  \end{equation}
as $n\to\infty$.
\end{Lemma}
\colb

\begin{proof}[Proof of Lemma~\ref{V_prob}]
Let $\epsilon>0$. With
$\{\tilde{\tau}_n\}_{n \geq 1}$
the stopping times 
as in Theorem~\ref{stop_time},
denote $ \tau_n = \tilde{\tau}_n \wedge T $.
Then there exists $N_0$ such that $\mathbb{P}(\tau_{N_0} < T ) \leq {\epsilon}/{4}$. 
Now, choose an $s$ such that $\mathbb{P}(\Omega_s) > 1 - {\epsilon}/{2}$, where $\Omega_s$ is as 
in Theorem~\ref{stop_time}. By Theorem~\ref{stop_time}, we have
  \begin{equation}
   \lim_{n \rightarrow \infty}
       \mathbb{E}\left[ \mathbbm{1}_{\Omega_s} \sup_{t \in [0,\tau_{N_0}]} \lVert u - u^n \rVert_V^2\right] = 0
  \end{equation}
which implies the convergence in probability, i.e.,
  \begin{equation}
   \lim_{n \rightarrow \infty} \mathbb{P} \left( \mathbbm{1}_{\Omega_s} \sup_{t \in [0,\tau_{N_0}]} \lVert u - u^n \rVert_V^2 \ge \delta \right) = 0,
  \end{equation}
for any $\delta > 0$. Hence, we have
  \begin{align}
   &\mathbb{P}\left(\mathbbm{1}_{\Omega_s}\sup_{t \in [0,T]} \lVert u - u^n \rVert_V^2 \geq \delta\right) 
   \nonumber\\&\indeq
   = \mathbb{P}\left(\{\sup_{t \in [0,T]} \lVert u - u^n \rVert_V^2 \geq \delta\}  \cap  \{\tau_{N_0} < T\} \cap  \{\omega \in \Omega_s\} \right) 
   \nonumber \\& \indeq\indeq
   + \mathbb{P}\left(\{\sup_{t \in [0,T]} \lVert u - u^n \rVert_V^2 \geq \delta \}  \cap  \{\tau_{N_0} = T\}  \cap  \{\omega \in \Omega_s\} \right)
   \nonumber \\&\indeq
   \leq \mathbb{P}\left( \tau_{N_0} < T \right) + \mathbb{P}\left(\mathbbm{1}_{\Omega_s}\sup_{t \in [0,\tau_{N_0}]} \lVert u - u^n \rVert_V^2 \ge \delta \right)
  \end{align}
and thus
  \begin{align}
   \mathbb{P}\left( \sup_{t \in [0,T]} \lVert u - u^n \rVert_V^2 \geq \delta \right) &\leq \mathbb{P} \left( \tau_{N_0} < T \right) + \mathbb{P} \left( \mathbbm{1}_{\Omega_s}\sup_{t \in [0,\tau_{N_0}]} \lVert u - u^n \rVert_V^2 \ge \delta \right) + \mathbb{P}\left( \Omega_s^c \right)
   \nonumber \\&\leq
   \frac{\epsilon}{4} + \frac{\epsilon}{4} + \frac{\epsilon}{2} = \epsilon
  \end{align}
for $n$ sufficiently large, and the proof is concluded.
\end{proof}

\begin{proof}[Proof of Theorem~\ref{T02}]
From the infinite dimensional version of It\^o's lemma we get
  \begin{align}
   &d(\phi(\lVert u \rVert_V^2)) + 2\nu\phi'(\lVert u \rVert_V^2)\lVert A u \rVert_H^2dt 
   \nonumber\\&\indeq
   = \phi' (\lVert u \rVert_V^2) \Big( 2\langle f,Au \rangle
   - 2\langle B(u,u),A u \rangle + \phi'(\lVert u \rVert_V^2)\lVert g(u) \rVert_{\VV}^2 \Big)dt 
   \nonumber\\&\indeq\indeq
   + 2\phi''(\lVert u \rVert_V^2) \sum_k \bigl\langle g_k(u),A u \bigr\rangle^2dt 
   + 2\phi'(\lVert u  \rVert_V^2)\bigl\langle g(u),A u \bigr\rangle dW
   \period
  \end{align}
We take the supremum up to the stopping time $\tilde{\tau}_m = \tau_m \wedge T$, where $\tau_m$ is introduced in Theorem~\ref{stop_time}. Denoting $\Omega_m = \{ \omega \in \Omega: \tilde{\tau}_m = T  \}$, we see that $\Omega_m \uparrow \Omega$ as $m \rightarrow \infty$ by Theorem~\ref{stop_time}. By taking the expectation on $\Omega_m$ and, suppressing $\mathbbm{1}_{\Omega_m}$ for simplicity of notation, we get
  \begin{align}
   \label{eqn_1}
   &\mathbb{E}\left[\sup_{[0,\tilde{\tau}_m]}\phi(\lVert u \rVert_V^2)\right] 
   + 2\nu\mathbb{E}\left[\int_0^{\tilde{\tau}_m} \phi'(\lVert u \rVert_V^2)\lVert A u \rVert_H^2ds\right]
   \nonumber\\&\indeq 
   \leq \phi'(\lVert u_0 \rVert_V^2) + \mathbb{E}\left[ \int_0^{\tilde{\tau}_m} 
     (T_1 + T_2 + T_3 + T_4)  ds \right] 
   + 2\mathbb{E}\left[\sup_{s\in [0,\tilde{\tau}_m]} \left \lvert \int_0^s T_0 dW_s \right \rvert\right]
  \end{align}
where we denoted
  \begin{align}
   T_0 &= 2 \phi^{\prime}(\lVert u \rVert_V^2) |\langle g(u),A u \rangle| \\
   T_1 &= 2 \phi^{\prime}(\lVert u \rVert_V^2)|\langle B(u,u),Au \rangle| \\
   T_2 &= 2 \phi^{\prime}(\lVert u \rVert_V^2)| \langle f, Au \rangle | 
   \leq 2\phi^{\prime}(\lVert u \rVert_V^2)\lVert f \rVert_H \lVert Au \rVert_H
   \leq C\phi^{\prime}(\lVert u \rVert_V^2)\lVert f \rVert_H^2 + \frac{\nu}{8}\phi^{\prime}(\lVert u \rVert_V^2) \lVert Au \rVert_H^2\\
   T_3 &= \phi^{\prime}(\lVert u \rVert_V^2) \lVert g(u) \rVert_{\VV}^2
   \leq C \phi^{\prime}(\lVert u \rVert_V^2) (1 + \lVert u \rVert_V^2)\\
   T_4 &= 2| \phi^{\prime\prime}(\lVert u \rVert_V^2) \langle g(u),A u \rangle|^2
   \leq C | \phi^{\prime\prime}( \lVert u \rVert_V^2 )| \lVert u \rVert_V^2( 1 + \lVert u \rVert_V^2)
  \end{align}
where $C$ is allowed to depend on $K_j$, for $j=0,1,2$, and $K_Y$.
Appealing to the BDG inequality, we have
  \begin{equation}
   \mathbb{E}\left[\sup_{s \in [0,\tau_m] } \left \lvert \int_0^s T_0dW \right \rvert \right]
   \leq C\mathbb{E}\left[\left( \int_0^{\tau_m} \bigl\lvert \phi'(\lVert u \rVert_V^2) \bigr\rvert^2 \lVert g(u) \rVert_{\VV}^2 \lVert u \rVert_V^2 ds \right)^{1/2}\right]
  \end{equation}
and thus, using the
Lipschitz condition on $g(u)$,
  \begin{align}
   &\mathbb{E}\left[\sup_{s \in [0,\tau_m]} \left\lvert \int_0^s T_0 dW \right\rvert\right]
   \leq C\mathbb{E}\left[\left( \int_0^{\tau_m} \frac{1}{(1 + \lVert u \rVert_V^2)^2} ( 1 +  \lVert u \rVert_V^2 ) \lVert u \rVert_V^2 ds   \right)^{1/2}\right]
   \leq C(T)
   \period
  \end{align}
Next, we estimate the term $T_1$ as
  \begin{align}
   T_1 &= 2 \phi^{\prime}(\lVert u \rVert_V^2)  \lvert\langle B(u,u),A u \rangle \rvert 
   \\ \nonumber 
   &\leq 2 \phi^{\prime}(\lVert u \rVert_V^2) \lVert u \rVert_H^{1/2} \lVert u \rVert_V^{1/2} \lVert u \rVert_V^{1/2}\lVert Au \rVert_H^{3/2}
   \\ \nonumber 
   &\leq C\phi^{\prime} (\lVert u \rVert_V^2) \lVert u \rVert_H^2\lVert u \rVert_V^4 + \frac{1}{4}\phi^{\prime}(\lVert u \rVert_V^2)\lVert Au \rVert_H^2
   \\ \nonumber 
   &\leq C\lVert u \rVert_H^2\lVert u \rVert_V^2 + \frac{1}{4}\phi^{\prime}(\lVert u \rVert_V^2)\lVert Au \rVert_H^2,
  \end{align}
where we note that by 
\eqref{EQ25}
  \begin{equation}
   \mathbb{E}\left[\int_0^T \lVert u \rVert_V^2\lVert u \rVert_H^2 dt \right]
   \leq M(\lVert u_0 \rVert_H^4,\lVert f \rVert_{V'}^4,T)
   \period
  \end{equation}
By combining all the estimates and writing out $\mathbbm{1}_{\Omega_m}$ explicitly, we obtain
  \begin{equation}
   \mathbb{E} \left[\mathbbm{1}_{\Omega_m} \sup_{[0,\tilde{\tau}_m]} \phi(\lVert u  \rVert_V^2)\right] \leq C(f,g,u_0,T)
   \period
  \end{equation}
By letting $m \rightarrow \infty$ and appealing to the monotone convergence theorem, we get
  \begin{equation}
   \mathbb{E} \left[\sup_{[0,T]} \phi(\lVert u  \rVert_V^2)\right] \leq C(f,g,u_0,T)
  \end{equation}
and the proof is concluded.
\end{proof}

\cole
\begin{Lemma} 
\label{L01}
Let $u^n$ be 
as in Definition~\ref{D02}. Then we have
  \begin{equation}
   \mathbb{E}\left[\sup_{[0,T]} \log( 1 + \lVert u^n \rVert_V^2)\right] \leq C(f,g,u_0,T)
   \label{EQ08}
  \end{equation}
and
  \begin{equation}
   \mathbb{E}\left[\sup_{[0,T]}  \log(1 + \lVert u - u^n \rVert_V^2)\right] \leq C(f,g,u_0, T),
   \label{EQ09}
  \end{equation}
for all $n\in{\mathbb N}$.
\end{Lemma}
\colb

{\begin{proof}[Proof of Lemma~\ref{L01}]
The proof 
of \eqref{EQ08}
follows the same steps as the proof of
Theorem~\ref{T02} and it is thus omitted.
The inequality~\eqref{EQ09} is a consequence of 
\eqref{EQ10}
and
\eqref{EQ08}.
\end{proof}

%

Now, we are ready to prove the first stated main result, Theorem~\ref{T01}.


\begin{proof}[Proof of Theorem~\ref{T01}]
Let $\epsilon\in(0,1)$.
By \eqref{EQ09}, we have
  \begin{equation}
   \sup_{[0,T]}  \log(1 + \lVert u - u^n \rVert_V^2)^{1-\epsilon} \rightarrow 0
   \label{EQ03}
  \end{equation}
in probability as $n \rightarrow \infty$. Moreover, using Lemma~\ref{L01},
  \begin{equation}
   \mathbb{E}\left[ \sup_{[0,T]}  (\log(1 + \lVert u - u^n \rVert_V^2)\right] \leq M(u_0,f,g,T)
   \period
   \label{EQ06}
  \end{equation}
Denoting
  \begin{equation}
   U_n
   =
   \sup_{[0,T]}  \log(1 + \lVert u - u^n \rVert_V^2)^{1-\epsilon}
   \label{EQ04}
  \end{equation}
we have by \eqref{EQ06}
  \begin{equation}
   \mathbb{E}\left[ U_n^{1/(1-\epsilon)}\right] \leq M(u_0,f,g,T)   
   \label{EQ05}
  \end{equation}
while \eqref{EQ09} gives
  \begin{equation}
   U_{n}^{1/(1-\epsilon)}
   \to 0
   \label{EQ07}
  \end{equation}
in probability.
Using 
the de~la~Vall\'ee-Poussin
criterion for uniform integrability (see e.g.~\cite{D13}), we get that
$U_n\to 0$ in $L^1$
as $n \rightarrow \infty$ and Theorem~\ref{T01} is proven.
\end{proof}

\section*{Acknowledgments}
We would like to thank R. Mikulevicius for many useful comments and
discussions. I.K. was supported in part by the NSF grant DMS-1311943,
while K.U. and M.Z. were supported in part by the NSF grant
DMS-1109562.

\nnewpage

\end{document}